\documentclass[11pt]{article}

\newcommand{\NP}{{\ensuremath{\mathrm{NP}}}}

\usepackage{amsmath,amssymb,amsthm}
\usepackage{etoolbox}
\usepackage{enumerate}
\usepackage{cleveref}
\providetoggle{long}
\usepackage[none]{hyphenat}
\settoggle{long}{true}

\usepackage{bbm}
\usepackage{graphicx}
\usepackage{fullpage}
\usepackage{a4wide}
\usepackage[usenames,dvipsnames]{color}
\usepackage{amssymb,amsmath,eepic,graphics,color,enumerate}
\usepackage{epsf}
\usepackage{graphics}
\usepackage{xcolor}
\usepackage{epsfig}
\usepackage{caption}
\usepackage{subcaption}
\usepackage{relsize}
\usepackage{rotating}
\usepackage{amsmath,enumerate,cite}

\newtheorem{theorem}{Theorem}[section]
\newtheorem{lemma}[theorem]{Lemma}
\newtheorem{corollary}[theorem]{Corollary}
\newtheorem{proposition}[theorem]{Proposition}
\newtheorem{observation}[theorem]{Observation}

\newtheorem{definition}[theorem]{Definition}

\title{Detecting strong cliques}

\author{Ademir Hujdurovi\'c\\
\small University of Primorska, IAM, Muzejski trg 2, SI6000 Koper, Slovenia\\
\small University of Primorska, FAMNIT, Glagolja\v ska 8, SI6000 Koper, Slovenia\\
\small \texttt{ademir.hujdurovic@upr.si}\\
\and
Martin Milani\v c\\
\small University of Primorska, IAM, Muzejski trg 2, SI6000 Koper, Slovenia\\
\small University of Primorska, FAMNIT, Glagolja\v ska 8, SI6000 Koper, Slovenia\\
\small \texttt{martin.milanic@upr.si}
\and
Bernard Ries\\
\small University of Fribourg, Department of Informatics, Bd de P\'erolles 90\\ \small CH-1700 Fribourg, Switzerland\\
\small \texttt{bernard.ries@unifr.ch}\\
}

\date{\today}

\begin{document}
\maketitle
\begin{abstract}
A strong clique in a graph is a clique intersecting every maximal independent set. We study the computational complexity of six algorithmic decision problems related to strong cliques in graphs
and almost completely determine their complexity in the classes of chordal graphs, weakly chordal graphs, line graphs and their complements, and graphs of maximum degree at most three. Our results rely on connections with matchings and relate to several graph properties studied in the literature, including well-covered graphs, localizable graphs, and general partition graphs.

\end{abstract}

%\newpage

\section{Introduction}

\subsection{Background}\label{sec:background}

A \emph{strong clique} in a graph $G$ is a clique intersecting every maximal independent set in $G$. Replacing the graph $G$ with its complement $\overline{G}$ maps the strong cliques of $G$ into the equivalent concept of \emph{strong independent sets} of $\overline{G}$, that is, independent sets intersecting every maximal clique in $\overline{G}$. The notions of strong cliques and strong independent sets in graphs played an important role in the study of perfect graphs and their subclasses (see, e.g.,~\cite{MR715895,MR778749,MR778765,MR888682,MR0439682}). Moreover, various other graph classes studied in the literature can be defined in terms of properties involving strong cliques~(see, e.g.,~\cite{MR3141630,MR3575013,MR3278773,HMR2018,MR2080087,MR1212874}). In some cases, strong cliques can be seen as a generalization of perfect matchings: transforming any regular triangle-free graph to the complement of its line graph maps any perfect matching into a strong clique (see~\cite{MR3488933,MT} for applications of this observation).

In this paper we report on advances regarding the computational complexity of several problems related to strong cliques in graphs. We pay particular attention to the class of graphs whose vertex set can be partitioned into strong cliques. Such graphs were called \emph{localizable} by Yamashita and Kameda~in 1999~\cite{MR1715546} and studied further by Hujdurovi\'c et al.~in 2018~\cite{HMR2018}. Localizable graphs form a rich class of \emph{well-covered graphs}, that is, graphs in which all maximal independent sets are of the same size. Moreover, localizable graphs and well-covered graphs coincide within the class of perfect graphs (see, e.g.,~\cite{HMR2018}). Well-covered graphs were introduced by Plummer in 1970~\cite{MR0289347} and have been studied extensively in the literature (see, e.g., the surveys by Plummer~\cite{MR1254158} and Hartnell~\cite{MR1677797}).

\begin{sloppypar}
Relatively few complexity results regarding problems involving strong cliques are available in the literature. Zang~\cite{MR1344757} showed that it is co-\NP-complete to test whether a given clique in a graph is strong and Ho{\`a}ng~\cite{MR1301855} established \NP-hardness of testing whether a given graph contains a strong clique. Hujdurovi\'c et al.~\cite{HMR2018} showed that recognizing localizable graphs is \NP-hard, even in the class of well-covered graphs whose complements are localizable. Polynomial-time recognition algorithms for localizable graphs are known within several classes of graphs, including triangle-free graphs, \hbox{$C_4$-free graphs}, and line graphs (see~\cite[Section 4.1]{HMR2018} for an overview). Results of Ho{\`a}ng~\cite{MR888682} and Burlet and Fonlupt~\cite{MR778765} imply the existence of a polynomial-time algorithm for determining if in every induced subgraph of a given graph $G$, each vertex belongs to a strong clique.
On the other hand, to the best of our knowledge, the complexity status of recognizing several graph classes related to strong cliques is in general still open. In particular, this is the case for: (i) the problem of recognizing {\em strongly perfect graphs}~\cite{MR715895,MR778749}, introduced by Berge as graphs every induced subgraph of which has a strong independent set,
(ii) the problem of determining whether \emph{every} edge of a given graph is contained in a strong clique or, equivalently, the problem of recognizing \emph{general partition graphs} (see, e.g.,~\cite{MR2080087,MR1212874,MR2794315}), and
(iii) the problem of recognizing {\em CIS graphs}, defined as graphs in which \emph{every} maximal clique is strong, or, equivalently, as graphs in which every maximal independent set is strong (see, e.g.,~\cite{MR3141630,MR2489416,MR3278773,MR2755907,MR2496915}).
\end{sloppypar}

In this paper we report on several advances regarding the computational complexity of six decision problems related to strong cliques. The problems are formally stated as follows.
%\medskip
%\noindent{\bf Six problems related to strong cliques.}

\begin{center}
\fbox{\parbox{0.85\linewidth}{\noindent
{\sc Strong Clique}\\[.8ex]
\begin{tabular*}{.93\textwidth}{rl}
{\em Input:} & A graph $G$ and a clique $C$ in $G$.\\
{\em Question:} & Is $C$ strong?
\end{tabular*}
}}
\end{center}

\begin{center}
\fbox{\parbox{0.85\linewidth}{\noindent
{\sc Strong Clique Existence}\\[.8ex]
\begin{tabular*}{.93\textwidth}{rl}
{\em Input:} & A graph $G$.\\
{\em Question:} & Does $G$ have a strong clique?
\end{tabular*}
}}
\end{center}

\begin{center}
\fbox{\parbox{0.85\linewidth}{\noindent
{\sc Strong Clique Vertex Cover}\\[.8ex]
\begin{tabular*}{.93\textwidth}{rl}
{\em Input:} & A graph $G$.\\
{\em Question:} & Is every vertex of $G$ contained in a strong clique?
\end{tabular*}
}}
\end{center}

\begin{center}
\fbox{\parbox{0.85\linewidth}{\noindent
{\sc Strong Clique Edge Cover} (a.k.a.~{\sc General Partitionability})\\[.8ex]
\begin{tabular*}{.93\textwidth}{rl}
{\em Input:} & A graph $G$.\\
{\em Question:} & Is every edge of $G$ contained in a strong clique?
\end{tabular*}
}}
\end{center}

\begin{center}
\fbox{\parbox{0.85\linewidth}{\noindent
{\sc Strong Clique Partition} \\[.8ex]
\begin{tabular*}{.93\textwidth}{rl}
{\em Input:} & A graph $G$ and a partition of its vertex set into
cliques.\\
{\em Question:} & Is every clique in the partition strong?
\end{tabular*}
}}
\end{center}

\begin{center}
\fbox{\parbox{0.85\linewidth}{\noindent
{\sc Strong Clique Partition Existence} (a.k.a.~{\sc Localizability})\\[.8ex]
\begin{tabular*}{.93\textwidth}{rl}
{\em Input:} & A graph $G$.\\
{\em Question:} & Can the vertex set of $G$ be partitioned into strong cliques?
\end{tabular*}
}}
\end{center}

In the rest of the paper, we will refer to the above six problems simply as the `six strong clique problems'.

\subsection{Motivation and overview of results}

\begin{sloppypar}
Our study has many sources of motivation. First, as mentioned above, problems {\sc Strong Clique}, {\sc Strong Clique Existence}, and {\sc Strong Clique Partition Existence} (a.k.a.~{\sc Localizability}) are known to be \NP-hard (see~\cite{MR1344757}, \cite{MR1301855}, and~\cite{HMR2018}, respectively). On the other hand, to the best of our knowledge, the complexity of the remaining three problems has not been determined in the literature. One of the aims of this paper is to partially close this gap and, more generally, to obtain further complexity and algorithmic results on the six strong clique problems in specific classes of graphs.
\end{sloppypar}

\begin{sloppypar}
Our work is also motivated by the general quest for finding further applications of connections between strong cliques and other graph theoretic concepts, in particular:
%. In particular, we have in mind two types of connections:
(i) of the fact that localizable graphs and well-covered graphs coincide within the class of perfect graphs, as observed in~\cite{HMR2018}, and (ii) of connections between strong cliques and matchings -- which naturally appear in the study of line graphs and their complements\hbox{~\cite{MR3488933,MR3575013,HMR2018,MR3162288,MT}}.

Regarding (i), we further elaborate on the approach used by Chv\'atal and Slater~\cite{MR1217991} and Sankaranarayana and Stewart~\cite{MR1161178} for proving hardness of recognizing well-covered graphs within the class of weakly chordal graphs (which are perfect), to derive similar results for {\sc Strong Clique Vertex Cover} and {\sc Strong Clique Partition}, two of the three problems from the above list whose complexity status was unknown prior to this work. These results are presented in~\Cref{sec:weakly-chordal}.

Many of our results motivated by (ii) are based on the study of {\sc Strong Clique Extension}, an auxiliary problem which we introduce in~\Cref{sec:a-problem}. We show that in any class of graphs, each of the first five of the six strong clique problems (as listed above) polynomially reduces to {\sc Strong Clique Extension}.
Results motivated by (ii) can be roughly divided into two parts.
\begin{itemize}
  \item We show that {\sc Strong Clique Extension} in either line graphs or their complements polynomially reduces to the maximum-weight matching problem. This implies polynomial-time solvability of the first five strong clique problems from the above list in the classes of line graphs and their complements. This is a significant generalization of the known fact that one of these problems, {\sc Strong Clique Edge Cover} (a.k.a.~{\sc General Partitionability}), is polynomial-time solvable in the class of line graphs, as proved by Levit and Milani\v{c}~\cite{MR3162288}. Exploiting further the connection with matchings, we show that
      {\sc Strong Clique Partition Existence} (a.k.a.~{\sc Localizability}) is \NP-complete in the class of complements of line graphs.
      Together with a result from Hujdurovi\'c et al.~\cite{HMR2018}, our results completely determine the computational complexity of the six strong clique problems from the above list in the classes of line graphs and their complements. These results are presented in~\Cref{sec:co-line}.

  \item We show that when restricted to subcubic graphs, {\sc Strong Clique Partition Existence} (a.k.a.~{\sc Localizability}) can be reduced to testing the existence of a perfect matching in a derived graph. We further complement this result, presented in~\Cref{sec:max3}, with a polynomial-time algorithm for {\sc Strong Clique Extension}, and consequently for each of the remaining five strong clique problems, in any class of graphs of bounded maximum degree, or, more generally, of bounded clique number.
\end{itemize}
\end{sloppypar}

\begin{sloppypar}
Together with some previous results, our approach using the {\sc Strong Clique Extension} problem also leads to a complete classification of the complexity of the six strong clique problems in the class of chordal graphs, a well-known subclass of weakly chordal graphs. In contrast with the hardness results for the class of weakly chordal graphs, all the six strong clique problems are solvable in polynomial time in the class of chordal graphs (see~\Cref{sec:a-problem}).
\end{sloppypar}

\begin{sloppypar}
Finally, our work is motivated by the question about characterizations of localizable graphs raised by Yamashita and Kameda~\cite{MR1715546}.
One of our results (\Cref{thm:complements-of-line-of-triangle-free}) shows that the property of localizability is difficult to recognize not only in the classes of weakly chordal graphs and well-covered graphs, as shown in~\cite{HMR2018}, but also in the class of complements of line graphs. Furthermore, as a corollary of our approach presented in~\Cref{sec:co-line}, we show that {\sc Strong Clique Partition Existence} (a.k.a.~{\sc Localizability}) is \NP-complete in the class of graphs of independence number $k$, for every fixed $k\ge 3$. On the positive side, we enhance our polynomial-time algorithm for testing localizability in the class of subcubic graphs by showing that within cubic graphs, the property coincides with the condition that every vertex is in a strong clique and classify graphs with these properties (Theorem \ref{thm:3-regular-improved}).
\end{sloppypar}

\medskip

In \Cref{table-results} we summarize the known and newly obtained complexity results for the six strong clique problems in general graphs and in special graph classes studied in this paper.

\begin{table}[h!]
%\begin{sidewaystable}
\makebox[\textwidth][c]{
 \footnotesize
\renewcommand{\arraystretch}{1.2}
\begin{tabular}{|l|c|c|c|c|c|c|}
  \hline
& all & chordal & weakly chordal & line & complements & subcubic \\
& graphs & graphs & graphs & graphs & of line graphs & graphs \\
  \hline
  % after \\: \hline or \cline{col1-col2} \cline{col3-col4} ...
  {\sc SC} & coNPc~\cite{MR1344757} & P$^\ast$ & coNPc~\cite{MR1344757} & {P} & {P} & P \\
  \hline
  {\sc SC Existence} & \NP-hard~\cite{MR1301855} & P$^\ast$ & {\NP-hard} & {P} & {P} &  P \\
  \hline
  {\sc SC Vertex Cover} & {\NP-hard} & P$^\ast$ & {\NP-hard} & {P} & {P} & P\\
  \hline
  {\sc SC Edge Cover} & ? & P$^\ast$~\cite{MR1477536}  & ? & P~\cite{MR3162288} & {P} & P \\
  \hline
  {\sc SC Partition} & {coNPc} & P$^\ast$  & {coNPc} & {P} & {P} &  P\\
  \hline
  {\sc SC Partition Existence} & coNPc~\cite{HMR2018} & P~\cite{HMR2018,MR1368737} & coNPc~\cite{HMR2018}  & P~\cite{HMR2018} & {\NP-complete} &  P \\
  \hline
\end{tabular}
}
\caption{Summary of complexity results for the six strong clique problems in various graph classes. SC stands for {\sc Strong Clique}, P for polynomial-time solvable, and coNPc for co-\NP-complete. Unreferenced entries mark results obtained in this paper. Entries for chordal graphs marked with an asterisk follow from~\Cref{cor:C4-free}. Results for weakly chordal graphs are presented in~\Cref{sec:weakly-chordal}.
Results for line graphs and their complements are presented in~\Cref{sec:co-line}. Results for subcubic graphs are presented in~\Cref{sec:max3}.}\label{table-results}
\end{table}
%\end{sidewaystable}

%====================================================================================

\section{Preliminaries}\label{sec:preliminaries}

In this section we recall some definitions, fix the notation, and
state some known results on line graphs and on localizable graphs.
We consider only finite, simple, and undirected graphs. Given a graph $G=(V,E)$, its complement $\overline{G}$ is the graph with vertex set $V$ in which two distinct vertices are adjacent if and only if they are non-adjacent in $G$. By $K_n$, $P_n$, and $C_n$ we denote as usual the $n$-vertex complete graph, path, and cycle, respectively, and by $K_{m,n}$ the complete bipartite graph with parts of sizes $m$ and $n$. The \emph{degree} of a vertex $v$ in a graph $G$ is denoted by $d_G(v)$, its neighborhood by $N_G(v)$ (or simply by $N(v)$ if the graph is clear from the context), and its closed neighborhood by $N_G[v]$ (or simply by $N[v]$). For a set of vertices $X\subseteq V(G)$, we denote by $N_G(X)$ (or $N(X)$) the set of all vertices in $V(G)\setminus X$ having a neighbor in $X$. The fact that two graphs $G$ and $H$ are isomorphic will be denoted by $G \cong H$. The minimum (resp.~maximum) degree of a vertex in a graph $G$ is denoted by $\delta(G)$ (resp.,~$\Delta(G)$). A graph is \emph{$k$-regular} if $\delta(G) = \Delta(G) = k$, and \emph{regular} if it is $k$-regular for some $k$. A graph is \emph{cubic} if it is $3$-regular and \emph{subcubic} if it is of maximum degree at most $3$.
The disjoint union of two vertex-disjoint graphs $G$ and $H$, denoted by $G+H$, is defined as the graph with vertex set $V(G)\cup V(H)$ and edge set $E(G)\cup E(H)$.
%We also write $nG$ for the disjoint union of $n$ copies of $G$.

For a graph $H$, we say that a graph $G$ is \emph{$H$-free} if $G$ has no induced subgraph isomorphic to $H$. A graph is said to be \emph{chordal} if it does not contain any induced cycle of length at least four. The \emph{line graph} of a graph $G=(V,E)$ is the graph $L(G)$ with vertex set $E$ in which two distinct vertices are adjacent if and only if they have a common endpoint as edges of $G$. The following result about line graphs will be useful in~\Cref{sec:co-line}.

\begin{theorem}[Roussopoulos~\cite{MR0424435} and Lehot~\cite{MR0347690}]\label{thm:line}
There exists a linear-time algorithm that tests if a given graph $G$ is (isomorphic to) the line graph of some graph $H$ and, if this the case, computes a graph $H$ such that $G \cong L(H)$.
%There exists a linear-time algorithm that tests if a given graph $G$ is the line graph of some graph $H$ and, if this the case, computes a graph $H$ such that $G = L(H)$.
\end{theorem}

For a positive integer $k$, a \emph{$k$-coloring} of $G$ is a mapping $c: V\rightarrow\{1,2,\ldots,k\}$ such that $c(u)\neq c(v)$ whenever $uv\in E$.  The \emph{chromatic number} $\chi(G)$ is the smallest integer $k$ for which $G$ has a $k$-coloring. A subset of edges $M\subseteq E$ is called a \textit{matching} if no two edges of $M$ share a common end-vertex. The \emph{chromatic index} of a graph $G$ is denoted by $\chi'(G)$ and defined as the smallest number of matchings of $G$ the union of which is $E$; clearly,
$\chi'(G)= \chi(L(G))$. A matching $M$ in a graph $G$ is \emph{perfect} if every vertex of $G$ is an endpoint of an edge in $M$. For undefined terms related to matchings, we refer the reader to~\cite{MR859549}.

Let $G=(V,E)$ be a graph. A subset $C\subseteq V$ is called a \emph{clique} in $G$ if any two vertices in $C$ are adjacent to each other.  The \emph{clique number} $\omega(G)$ is the number of vertices in a maximum clique of $G$. A \emph{triangle} is the graph $K_3$. We will often identify a triangle with the set of its edges; whether we consider a triangle as a set of vertices or as a set of edges will always be clear from the context. A subset $I\subseteq V$ is called an \emph{independent set} in $G$ if any two vertices in $I$ are non-adjacent to each other.  The \emph{independence number} $\alpha(G)$ is the number of vertices in a maximum independent set of $G$. Given a clique $C$ in $G$ and a set $S\subseteq V(G)\setminus C$, we say that clique $C$ is \emph{dominated by $S$} if $C\subseteq N_G(S)$.
The following is an easy but useful characterization of strong cliques (see, e.g., remarks following~\cite[Theorem 2.3]{HMR2018}).

\begin{lemma}\label{lem:strong-cliques}
A clique $C$ in a graph $G$ is not strong if and only if it is
dominated by an independent set $I \subseteq V(G)\setminus C$.
\end{lemma}

A set $D\subseteq V$ is called \emph{dominating} if every vertex in $V\setminus D$ has at least one neighbor in $D$. If $D$ is in addition an independent set, it is called an \emph{independent dominating set}. We denote by $i(G)$ the \emph{independent domination number} of $G$, that is, the minimum size of an independent dominating set in $G$. Furthermore, we denote by $\theta(G)$ the \emph{clique cover number} of $G$, that is, the minimum number of cliques that partition its vertex set. It is well-known that every graph $G$ has $\alpha(G) = \omega(\overline{G})$, $\theta(G) = \chi(\overline{G})$, $\chi(G)\ge \omega(G)$, and $\theta(G)\ge \alpha(G)$. It follows that every graph $G$ satisfies the following chain of inequalities:
\begin{equation*}
i(G)\le \alpha(G)\le \theta(G)\,.
\end{equation*}
Clearly, a graph $G$ is well-covered if and only if $i(G) = \alpha(G)$.

For a positive integer $k$, we say that a graph $G$ is {\em $k$-localizable} if it admits a partition of its vertex set into $k$ strong cliques. A graph $G$ is said to be \emph{semi-perfect} if $\theta(G)= \alpha(G)$, that is, if there exists a collection of $\alpha(G)$ cliques partitioning its vertex set. We will refer to such a collection as an \emph{$\alpha$-clique cover} of $G$. Thus, $G$ is semi-perfect if and only if it has an $\alpha$-clique cover, and $G$ is $\alpha(G)$-localizable if and only if it has a $\alpha$-clique cover in which every clique is strong.

In the following theorem we recall several equivalent formulations of localizability.

\begin{theorem}[Hujdurovi\'c et al.~\cite{HMR2018}]\label{prop:alpha-chi-bar}
For every graph $G$, the following statements are equivalent.
\begin{enumerate}
  \item[(a)] $G$ is localizable.
  \item[(b)] $G$ is $\alpha(G)$-localizable (equivalently, $G$ has an $\alpha$-clique cover in which every clique is strong).
  \item[(c)] $G$ has an $\alpha$-clique cover and every clique in every $\alpha$-clique cover of $G$ is strong.
  \item[(d)] $G$ is well-covered and semi-perfect.
  \item[(e)] $i(G) = \theta(G)$.
\end{enumerate}
\end{theorem}

A graph $G$ is \emph{perfect} if $\chi(H) = \omega(H)$ holds for every induced subgraph $H$ of $G$. Since the complementary graph $\overline{G}$ is perfect whenever $G$ is perfect~\cite{MR0302480}, every perfect graph $G$ is also semi-perfect. The following consequence of~\Cref{prop:alpha-chi-bar} will be useful.

\begin{corollary}\label{cor:perfect}
A perfect graph is well-covered if and only if it is localizable.
\end{corollary}

%====================================================================================

\section{A hardness proof and its implications}
\label{sec:weakly-chordal}

An important subclass of perfect graphs generalizing the class of chordal graphs is the class of weakly chordal graphs. A graph is \emph{weakly chordal} if neither $G$ nor its complement contain an induced cycle of length at least~$5$. Chv\'atal and Slater~\cite{MR1217991} and Sankaranarayana and Stewart~\cite{MR1161178} proved the following hardness result.

\begin{theorem}\label{thm:wc-rec-NP-hard}
The problem of recognizing well-covered graphs is co-\NP-complete, even for weakly chordal graphs.
\end{theorem}

Since weakly chordal graphs are perfect, \Cref{cor:perfect} implies that a weakly chordal graph is well-covered if and only if it is localizable. Thus,
\Cref{thm:wc-rec-NP-hard} yields the following.

\begin{sloppypar}
\begin{theorem}[Hujdurovi\'{c} et al.~\cite{HMR2018}]\label{thm:weakly-chordal}
{\sc Strong Clique Partition Existence} (a.k.a.~{\sc Localizability}) is co-\NP-complete in the class of weakly chordal graphs.
\end{theorem}
\end{sloppypar}

Recall that, as noted in \Cref{table-results}, it follows from results of Prisner et al.~\cite{MR1368737} that localizable graphs can be recognized in polynomial time within the class of chordal graphs (and, more generally, also in the class of $C_4$-free graphs~\cite{HMR2018}).

Both proofs of Theorem~\ref{thm:wc-rec-NP-hard} from~\cite{MR1217991,MR1161178} are based on a reduction from the $3$-SAT problem. The input to the $3$-SAT problem is a set $X = \{x_1, \ldots, x_n\}$ of Boolean variables and a collection $\mathcal{C} = \{C_1,\ldots, C_m\}$ of clauses of size $3$ over $X$. Each clause is a disjunction of exactly three \emph{literals}, where a literal is either one of the variables (say $x_i$) or its negation (denoted by $\overline{x_i}$). The task is to determine whether the formula $\varphi = \bigwedge_{i = 1}^mC_i$ is satisfiable, that is, whether there exists a truth assignment to the $n$ variables which makes all the clauses simultaneously evaluate to true. The reduction produces from a given $3$-SAT instance $I$ a weakly chordal graph $G(I)$ such that $G(I)$ is not well-covered if and only if the formula is satisfiable.
More precisely, given an input $(X,\mathcal{C})$ to the $3$-SAT problem, representing a formula $\varphi$, a graph $G = G(X,\mathcal{C})$ is constructed with vertex set $V(G) = C\cup L$, where $C = \{c_1,\ldots, c_m\}$ forms a clique, $L = \{x_1,\overline{x_1}, \ldots, x_n,\overline{x_n}\}$, for each $i\in \{1,\ldots, n\}$, vertices $x_i$ and $\overline{x_i}$ are adjacent, for each $c_i\in C$ and each literal $\ell\in L$, $c_i\ell\in E(G)$ if and only if clause $C_i$ contains literal $\ell$, and there are no other edges. Then, $\varphi$ is satisfiable if and only if $G$ is not well-covered~\cite{MR1217991,MR1161178}.

\begin{sloppypar}
A similar construction, using the complement of $G$ instead of $G$, was used in a reduction from the satisfiability problem by Zang~\cite{MR1344757} to show that the problem of testing whether a given independent set in a graph is strong is co-\NP-complete. Since the class of weakly chordal graphs is closed under taking complements, Zang's proof shows the following.
\end{sloppypar}

\begin{theorem}
{\sc Strong Clique} is co-\NP-complete in the class of weakly chordal graphs.
\end{theorem}

We now show that suitable modifications of the above reduction can be used to establish \NP-hardness of three of the six strong clique problems listed in~\Cref{sec:background}, even in the class of weakly chordal graphs.

\begin{sloppypar}
\begin{theorem}
\label{thm:strongcliques}
{\sc Strong Clique Existence} and {\sc Strong Clique Vertex Cover} are \NP-hard, even in the class of weakly chordal graphs. {\sc Strong Clique Partition} is co-\NP-complete, both in general and for weakly chordal graphs.
\end{theorem}
\end{sloppypar}

\begin{proof}
We use a slight modification of the above reduction.
First, note that the $3$-SAT problem remains \NP-complete on instances satisfying the following assumptions:
\begin{enumerate}[(i)]
  \item {\it No clause contains a pair of the form $\{x_i,\overline{x_i}\}$.} Indeed, if there exists a clause containing both $x_i$ and $\overline{x_i}$, then such a clause is always satisfied and can be removed from the instance to obtain an equivalent one.

  \item {\it Each literal appears in at least one clause.}
  Indeed, if some literal $\ell$ does not appear in any clause, then we can set that literal to false and remove each clause containing its negation to obtain an equivalent instance.

  \item {\it No variable $x_i$ is such that every clause contains either $x_i$ or $\overline{x_i}$.} Indeed, instances having a variable $x_i$ such that every clause contains either $x_i$ or $\overline{x_i}$ can be solved in polynomial time, by solving two instances of the $2$-SAT problem corresponding to setting $x_i$ to true (resp., to false) and obtained by eliminating all clauses containing $x_i$ (resp., $\overline{x_i}$) and deleting $\overline{x_i}$ (resp., $x_i$) from all clauses containing it. The $2$-SAT problem is defined similarly as the $3$-SAT problem, except that every clause is of size two. The problem is solvable in linear time~\cite{MR526451,MR0471974}.
\end{enumerate}
We may thus assume that we are given a $3$-SAT instance $I$ satisfying assumptions (i)--(iii). Let $X$, $\mathcal{C}$, $\varphi$, $G$, $C = \{c_1,\ldots, c_m\}$, and $L = \{x_1,\overline{x_1}, \ldots, x_n,\overline{x_n}\}$ be as above.

Assumption (i) implies that none of the cliques $\{x_i,\overline{x_i}\}$ in $G$ is contained in the neighborhood of a vertex in $C$. Since $N_G(\{x_i,\overline{x_i}\})$ is a clique, it follows that each clique $\{x_i,\overline{x_i}\}$ is strong in $G$.

Rephrasing the observation of Zang~\cite{MR1344757} in our setting, we claim that formula $\varphi$ is satisfiable if and only if clique $C$ is not strong. Indeed, on the one hand, literals in $L$ that are set to true in a satisfying assignment form an independent set disjoint from $C$ that dominates $C$, and~\Cref{lem:strong-cliques} applies. Conversely, if $C$ is not strong then any maximal independent set $I$ disjoint from $C$ necessarily consists of exactly one literal from each pair $\{x_i,\overline{x_i}\}$; setting all literals in $I$ to true yields a satisfying assignment.

It follows that each of the cliques in the partition $\{C,\{x_1,\overline{x_1}\}, \ldots, \{x_1,\overline{x_n}\}\}$ of $V(G)$ is strong if and only if $C$ is strong, if and only if the formula is not satisfiable. Consequently, {\sc Strong Clique Partition} is \NP-hard in the class of weakly chordal graphs. Membership in co-\NP~(in any class of graphs) follows from the observation that a short certificate of a no instance consists of a pair $(K,I)$ where $K$ is one of the cliques in the partition and $I$ is a maximal independent set disjoint from $K$.

Next, we prove that {\sc Strong Clique Vertex Cover} is \NP-hard in the class of weakly chordal graphs. Observe that each vertex $\ell\in L$ is contained in a strong clique (namely the $K_2$ containing the literal $\ell$ and its negation). Moreover, we will now show that no clique of $G$ containing both a vertex from $L$ and a vertex from $C$ is strong. Let $K$ be such a clique, with, say, $\ell\in L\cap K$ and $c_i\in C\cap K$. Let us denote by $\overline{\ell}$ the literal complementary to $\ell$. It follows from the above that $\{\ell, \overline{\ell}\}$ is strong and in particular that $\overline{\ell}$ is not in $K$. Therefore $K\cap L = \{\ell\}$. By assumption (iii) on the $3$-SAT instance, there is a clause $c_j$ that contains neither $\ell$ not $\overline{\ell}$. Then $c_j\neq c_i$ and the set $I = \{\overline{\ell}, c_j\}$ is an independent set disjoint from $G$ that dominates $K$. By~\Cref{lem:strong-cliques}, $K$ is not strong.

We know that every vertex $\ell\in L$ is contained in a strong clique and, since no clique of $G$ containing both a vertex from $L$ and a vertex from $C$ is strong, a vertex in $C$ is contained in a strong clique if and only if $C$ is strong. It follows that every vertex of $G$ is contained in a strong clique if and only if $C$ is strong, which, as argued above, is equivalent to
non-satisfiability of $\varphi$. The claimed \NP-hardness follows.

It remains to show that {\sc Strong Clique Existence} is \NP-hard in the class of weakly chordal graphs. To this end, take the graph $G$ constructed as above and modify it to a graph $G'$ by adding, for each $i\in \{1,\ldots, n\}$, four new vertices $u_i,v_i,\overline{u_i},\overline{v_i}$ and edges $x_iu_i$,
$x_iv_i$, $u_iv_i$, $u_i\overline{u_i}$,
$u_i\overline{v_i}$, $v_i\overline{u_i}$,
$\overline{u_i}\,\overline{v_i}$,
$\overline{x_i}\,\overline{u_i}$, and
$\overline{x_i}\,\overline{v_i}$ (see~\Cref{fig:reduction}).

\begin{figure}[!ht]
  \centering
  \includegraphics[width=\linewidth]{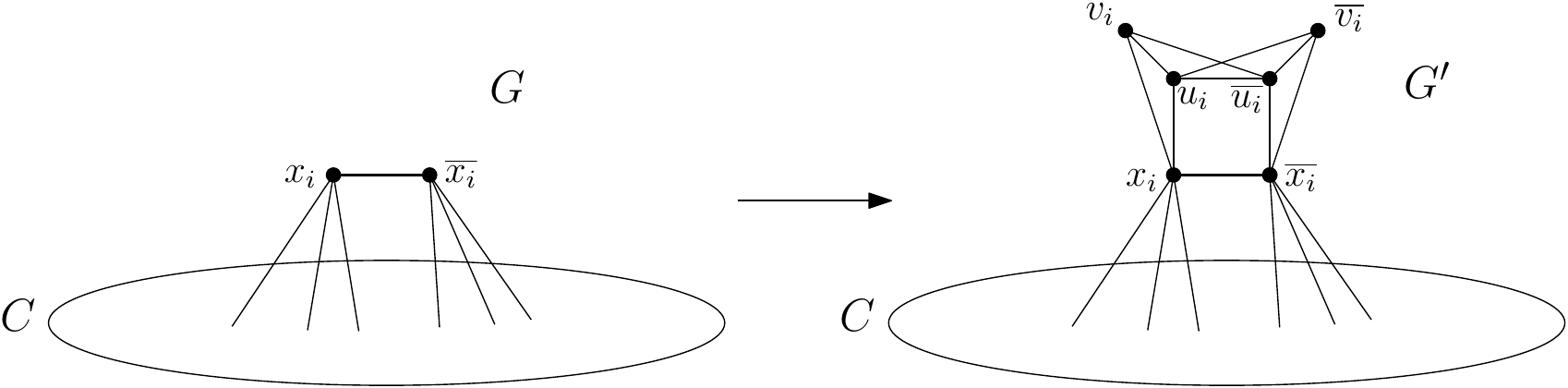}
  \captionof{figure}{Transforming $G$ into $G'$.}
  \label{fig:reduction}
\end{figure}

It is not difficult to see that the resulting graph $G'$ is weakly chordal.
Since none of the newly added vertices is adjacent to any vertex in $C$, the fact that $C$ is not strong if and only if formula $\varphi$ is satisfiable remains valid for $G'$.
Moreover, using~\Cref{lem:strong-cliques} it can be verified that every maximal clique of $G'$ other than $C$ is not strong:
\begin{itemize}
  \item The fact that no clique of $G$ containing both a vertex from $L$ and a vertex from $C$ is strong (in $G$) implies the same property for $G'$.
  \item For each $i\in \{1,\ldots, n\}$, the clique $\{x_i,\overline{x_i}\}$ is not strong since it is dominated by the independent set $\{v_i,\overline{v_i}\}$.
  \item For each $i\in \{1,\ldots, n\}$, the clique $\{x_i,u_i,v_i\}$ is not strong since it is dominated by the independent set $\{c_j,\overline{u_i}\}$, where $C_j$ is a clause containing $x_i$ (note that such a clause exists by assumption (ii)). A symmetric argument shows that  the clique $\{\overline{x_i},\overline{u_i},\overline{v_i}\}$ is not strong.
  \item For each $i\in \{1,\ldots, n\}$, the clique $\{u_i,\overline{u_i},\overline{v_i}\}$ is not strong since it is dominated by the independent set $\{x_i, v_i\}$. A symmetric argument shows that the clique $\{u_i,\overline{u_i}, v_i\}$ is not strong.
\end{itemize}
It follows that $G'$ contains a strong clique if and only if
$C$ is strong, which is equivalent to non-satisfiability of $\varphi$. This establishes the claimed \NP-hardness of {\sc Strong Clique Existence} in the class of weakly chordal graphs.
\end{proof}
\begin{sloppypar}
\Cref{thm:strongcliques} implies that the obvious certificate for yes instances of the problem of recognizing localizable graphs is most likely not verifiable in polynomial time, leaving the membership status of
{\sc Strong Clique Partition Existence} (a.k.a.~{\sc Localizability})
in NP (or in co-NP) unclear.
\end{sloppypar}

%Moreover, since a graph $G$ is weakly chordal if and only if its complement $\overline{G}$ is weakly chordal, \Cref{thm:strongcliques} has the following consequence.
%
%\begin{corollary}\label{cor:weakly-chordal}
%The following problems are \NP-hard, even for weakly chordal graphs:
%\begin{enumerate}
%  \item The problem of determining whether, given a graph and a partition of its vertex set into independent sets, each independent set in the partition is strong. (This problem is also co-\NP-complete.)
%  \item The problem of determining whether every vertex of a given graph is contained in a strong independent set.
%\end{enumerate}
%\end{corollary}
%
%{\color{blue}{\bf Question (M.):} Is the above corollary of sufficient interest to be kept? I am not completely sure.
%Note that one could have similar corollaries in \Cref{sec:co-line}.
%Perhaps it is best to remove it? It is straightforward anyway, whoever needs it should be able to observe it easily.}

%====================================================================================

\section{A useful auxiliary problem}\label{sec:a-problem}

We introduce one more problem related to strong cliques.

\begin{center}
\fbox{\parbox{0.85\linewidth}{\noindent
{\sc Strong Clique Extension}\\[.8ex]
\begin{tabular*}{.93\textwidth}{rl}
{\em Input:} & A graph $G$ and a clique $C$ in $G$.\\
{\em Question:} & Does $G$ contain a strong clique $C'$ such that
$C\subseteq C'$?
\end{tabular*}
}}
\end{center}

As the next lemma shows, {\sc Strong Clique Extension} in some sense generalizes the first five of the six strong clique problems listed in~\Cref{sec:background}.

\begin{sloppypar}
\begin{lemma}\label{lem:reduction}
Let $\mathcal{G}$ be a class of graphs such that
{\sc Strong Clique Extension} is solvable in polynomial time for graphs in
$\mathcal{G}$.
Then, each of the problems {\sc Strong Clique}, {\sc Strong Clique Existence}, {\sc Strong Clique Vertex Cover}, {\sc Strong Clique Edge Cover}, and {\sc Strong Clique Partition} is solvable in polynomial time for graphs in $\mathcal{G}$.
\end{lemma}
\end{sloppypar}

\begin{proof}
Let $G$ be a graph from $\mathcal{G}$ and let $C$ be a clique in $G$.
It is clear that $C$ is a strong clique in $G$ if and only if
$C$ is a maximal clique and $G$ contains a strong clique $C'$ such that
$C\subseteq C'$.
It follows that testing if $C$ is strong can be done in polynomial time by testing if $C$ is maximal and if $C$ extends to a strong clique. The {\sc Strong Clique Existence} problem is equivalent to determining if the empty clique extends to some strong clique. Similarly, the {\sc Strong Clique Vertex Cover} (resp., {\sc Strong Clique Edge Cover}) problem can be solved in polynomial time simply by testing if every clique of size one (resp., two) extends to a strong clique.
Finally, since {\sc Strong Clique} is polynomial-time solvable for graphs in $\mathcal{G}$, {\sc Strong Clique Partition} can also be solved in polynomial time for graphs in $\mathcal{G}$ by testing whether each of the clique in the given partition is strong.
\end{proof}

As an immediate application of \Cref{lem:reduction}, we observe that
together with a result from~\cite{HMR2018}, the lemma implies that all the above problems are polynomial-time solvable in the class of chordal graphs, and even more generally, in the class of $C_4$-free graphs.

\begin{proposition}
{\sc Strong Clique Extension} is solvable in polynomial time in the class of $C_4$-free graphs.
\end{proposition}

\begin{proof}
As shown in \cite[Lemma 4.4]{HMR2018}, a clique in a $C_4$-free graph is strong if and only if it is simplicial, that is, if it consists of a vertex and all its neighbors. Therefore, in the class of $C_4$-free graphs {\sc Strong Clique Extension} reduces to the problem of testing if a given clique is contained in a simplicial clique. Since the simplicial cliques of a graph can be listed in polynomial time, polynomial-time solvability of {\sc Strong Clique Extension} in the class of $C_4$-free graphs follows.
\end{proof}

\begin{corollary}\label{cor:C4-free}
Each of the problems {\sc Strong Clique}, {\sc Strong Clique Existence}, {\sc Strong Clique Vertex Cover}, {\sc Strong Clique Edge Cover}, and {\sc Strong Clique Partition} is solvable in polynomial time in the class of $C_4$-free graphs.
\end{corollary}

Further applications of~\Cref{lem:reduction} will be given in Sections~\ref{sec:co-line} and~\ref{sec:max3}.

\begin{sloppypar}
\section{Connections with matchings: line graphs and their complements}\label{sec:co-line}
\end{sloppypar}

\begin{sloppypar}
Relationships between strong cliques and matchings have already been used in a number of works studying properties of line graphs and their complements~\cite{MR3488933,MR3575013,HMR2018,MR3162288,MT}. We now further elaborate on these connections in order to obtain new algorithmic and complexity results on the six strong cliques problems in the classes of line graphs and their complements. Two cases have already been settled in the literature, both giving polynomial-time algorithms for the class of line graphs: one for {\sc Strong Clique Edge Cover} (a.k.a.~{\sc General Partitionability}), due to Levit and Milani\v{c}~\cite{MR3162288}, and one for {\sc Strong Clique Partition Existence} (a.k.a.~{\sc Localizability}), due to Hujdurovi\'c et al.~\cite{HMR2018}. A related result is also a result of Boros et al.~\cite{MR3575013} giving a polynomial-time algorithm to test whether a given line graph is CIS, that is, whether every maximal clique in a line graph is strong. Since the class of CIS graphs is closed under complementation, this result also yields a polynomial-time algorithm for the same problem in the class of complements of line graphs.
\end{sloppypar}

In this section, we show that {\sc Strong Clique Extension} is solvable in polynomial time in the class of line graphs and their complements. Together with the above-mentioned result from~\cite{HMR2018}, this implies that all six strong clique problems listed in~\Cref{sec:background} are polynomial-time solvable in the class of line graphs. For the class of complements of line graphs, however, we show that {\sc Strong Clique Partition Existence}
remains \NP-complete.

We first introduce some notation and translate the relevant concepts from the line graph or its complement to the root graph.

\begin{definition}\label{def:M_G}
Given a graph $G$, we denote by $\mathcal{T}_G$ the set $\{E(T) \mid T$ is a triangle in $G\}$, by $\mathcal{S}_G$ the set of all \emph{stars} of $G$, that is, sets of edges of the form $E(v) = \{e\in E(G)\mid v$ is an endpoint of $e\}$ for $v\in V(G)$, and by $\mathcal{M}_G$ the set of all inclusion-maximal sets in the family $\mathcal{T}_G\cup \mathcal{S}_G$.
\end{definition}

The following observation is an immediate consequence of the definitions.

\begin{observation}\label{obs}
For every graph $G$ without isolated vertices, the following holds:
\begin{enumerate}
\item $e$ is an edge of $G \Longleftrightarrow$ $e$ is a vertex of $L(G) \Longleftrightarrow$ $e$ is a vertex of $\overline{L(G)}$.
\item For every $F\subseteq E(G)$, we have:
\begin{itemize}
  {\setlength\itemindent{-0.5cm}\item $F$ is a clique in $L(G)$
$\Longleftrightarrow$
$F\subseteq E(v)$ for some $v\in V(G)$ or $F\in \mathcal{T}_G$.}
    {\setlength\itemindent{-0.5cm}\item $F$ is a maximal clique in $L(G)$
$\Longleftrightarrow$
$F$ is a maximal independent set in $\overline{L(G)}$
$\Longleftrightarrow$ $F\in \mathcal{M}_G$.}
  {\setlength\itemindent{-0.5cm}  \item $F$ is an independent set in $L(G) \Longleftrightarrow
  F$ is a clique in $\overline{L(G)} \Longleftrightarrow F$ is a matching in $G$.}
  {\setlength\itemindent{-0.5cm}    \item $F$ is a maximal independent set in $L(G)$
$\Longleftrightarrow$
$F$ is a maximal clique in $\overline{L(G)}$
$\Longleftrightarrow$
$F$ is a maximal matching in $G$.}
  {\setlength\itemindent{-0.5cm}    \item $F$ is a strong clique in $L(G)$
$\Longleftrightarrow$
$F$ is a set in $\mathcal{M}_G$ that intersects every maximal matching in $G$.}
  {\setlength\itemindent{-0.5cm}    \item $F$ is a strong clique in $\overline{L(G)}$
$\Longleftrightarrow$ $F$ is a maximal matching in $G$ that intersects every set in $\mathcal{M}_G$.}
\end{itemize}
\end{enumerate}
\end{observation}

\begin{theorem}\label{thm:SCEXT-line}
{\sc Strong Clique Extension} is solvable in polynomial time in the class of line graphs.
\end{theorem}

\begin{proof}
Let $G$ be a line graph and let $C$ be a clique in $G$. Using~\Cref{thm:line}, we can compute in linear time a graph $H$ such that $G \cong L(H)$. We also compute, in polynomial time, the set $\mathcal{M}_H$. Since every strong clique in a graph is a maximal clique and the maximal cliques of $G$ correspond exactly to the elements of $\mathcal{M}_H$, it suffices to show that there is a polynomial-time algorithm to check if a given $C'\in \mathcal{M}_H$ corresponds to a strong clique $C_G'$ in $G$. Then, to check if $C$ is contained in a strong clique it will suffice to check if any $C'\in \mathcal{M}_H$ such that $C\subseteq C_G'$ corresponds to a strong clique.

Let $C'\in \mathcal{M}_H$. Thus, $C'$ is a set of edges in $H$ and we need to check if $C'$ intersects every maximal matching in $H$.
Suppose first that $C'\in \mathcal{T}_H$. Then $C'$ is the edge set of a triangle $T$ in $H$. It is not difficult to see that $H$ contains a maximal matching disjoint from $C'$ if and only if $H-C'$ contains a matching
saturating at least two vertices of $T$. This is a local property that can be easily checked in polynomial time, for example by examining all pairs of disjoint edges in $H-C'$ connecting $V(T)$ with the rest of the graph.
Suppose now that $C'\in \mathcal{S}_H$.
Then there exists a vertex $v\in V(H)$ such that $C' = E(v)$, and $E(v)$ intersects every maximal matching in $H$ if and only if every maximal matching in $H$ saturates $v$. This fails if and only if $H-v$ contains a matching saturating all vertices in $N_H(v)$. There is a standard reduction of this problem to an instance of the maximum-weight matching problem in $H-v$: to each edge $e$ of $H-v$ assign the weight $w(e) = |e\cap N_H(v)|$ and verify if $H-v$ has a matching of weight at least $|N_H(v)|$. Since the maximum-weight matching problem is polynomial-time solvable (e.g., by Edmonds~\cite{MR0183532}), the theorem follows.
\end{proof}

\Cref{lem:reduction} and~\Cref{thm:SCEXT-line} imply the following.

\begin{sloppypar}
\begin{corollary}\label{cor:line}
{\sc Strong Clique}, {\sc Strong Clique Existence}, {\sc Strong Clique Vertex Cover}, {\sc Strong Clique Edge Cover}, and {\sc Strong Clique Partition} problems are
solvable in polynomial time in the class of line graphs.
\end{corollary}
\end{sloppypar}

We now turn to complements of line graphs.

\begin{sloppypar}
\begin{theorem}\label{thm:SCEXT-co-line}
{\sc Strong Clique Extension} is solvable in polynomial time in the class of complements of line graphs.
\end{theorem}
\end{sloppypar}

\begin{proof}
Let $G$ be the complement of a line graph and let $C$ be a clique in $G$. We first compute the complement of $G$ and then, using~\Cref{thm:line}, a graph $H$ such that $G \cong \overline{L(H)}$. Thus $C$ corresponds to a matching $C_H$ in $H$ and by~\Cref{obs}, it suffices to check if $H$ contains a matching $M$ such that $C_H\subseteq M$ and $M$ intersects every set in $\mathcal{M}_H$, where $\mathcal{M}_H$ is the set of all inclusion-maximal sets in the family of triangles and stars of $H$. We reduce this problem to an instance of the maximum-weight matching problem in $H'$ where $H'$ is the graph obtained from $H$ by deleting from it all edges not in $C_H$ that share an endpoint with an edge in $C_H$ (as no such edge can be part of a matching containing $C_H$).

Before explaining the reduction, we show that for every matching $M$ in $H'$ that intersects every set in $\mathcal{M}_H$, we have $C_H\subseteq M$. Indeed, if $e\in C_H$ then there exists a set $S\in \mathcal{M}_H$ such that $e\in S$; moreover, the construction of $H'$ implies that $S\cap E(H') = \{e\}$ and hence $M$ can only intersect $S$ in $e$; this implies $e\in M$. Since $e\in C_H$ was arbitrary, we have $C_H\subseteq M$, as claimed.

The reduction is as follows. We compute the set $\mathcal{M}_H$ and assign to each edge $e\in E(H')$ weight $w(e)$, which denotes the number of sets in $\mathcal{M}_H$ containing $e$. Note that no set in $\mathcal{M}_H$ contains two edges from the same matching. It follows that for any matching $M$ in $H'$, the total number of sets in $\mathcal{M}_H$ intersected by some edge in $M$ is equal to the total weight of $M$ (with respect to $w$). Consequently, $H'$ contains a matching containing $C_H$ that intersects every set in $\mathcal{M}_H$, if and only if $H'$ contains a matching that intersects every set in $\mathcal{M}_H$, if and only if
$H'$ contains a matching of total $w$-weight at least (equivalently, exactly) $k$, where $k$ is the cardinality of $\mathcal{M}_H$. The weight function $w(e)$ can be computed in polynomial time and the maximum-weight matching problem is also polynomial-time solvable, hence the theorem is proved.
\end{proof}

\Cref{lem:reduction} and~\Cref{thm:SCEXT-co-line} imply the following.

\begin{sloppypar}
\begin{corollary}\label{cor:co-line}
{\sc Strong Clique}, {\sc Strong Clique Existence}, {\sc Strong Clique Vertex Cover}, {\sc Strong Clique Edge Cover}, and
{\sc Strong Clique Partition} problems are
solvable in polynomial time in the class of complements of line graphs.
\end{corollary}
\end{sloppypar}

An immediate consequence of the fact that {\sc Strong Clique Partition} is solvable in polynomial time in the class of complements of line graphs is the following.

\begin{sloppypar}
\begin{corollary}\label{cor:NP}
{\sc Strong Clique Partition Existence} restricted to the class of complements of line graphs is in NP.
\end{corollary}
\end{sloppypar}

In contrast to~\Cref{cor:co-line}, we show next that
{\sc Strong Clique Partition Existence} is
\NP-complete in the class of complements of line graphs.
The proof is based on the following.

\begin{lemma}
\label{lem:complements-of-line-graphs}
Let $k\ge 2$ be an integer and let $G$ be a triangle-free $k$-regular graph.
Then $G$ is $k$-edge-colorable if and only if $\overline{L(G)}$ is localizable.
\end{lemma}

\begin{sloppypar}
\begin{proof}
First, note that since $G$ is $k$-regular, we have $\chi'(G) \ge k$. In particular, $G$ is $k$-edge-colorable if and only if $\chi'(G) = k$.
Since $G$ has no vertices of degree $0$ or $1$, no two stars of $G$ are comparable with respect to set inclusion, which implies that the maximal cliques of $L(G)$ are exactly the stars of $G$. Since $G$ is $k$-regular, all maximal cliques of $L(G)$ are of size $k$ and consequently $\overline{L(G)}$ is a well-covered graph of independence number $k$. It now follows from~\Cref{prop:alpha-chi-bar} that $\overline{L(G)}$ is localizable if and only if $\theta(\overline{L(G)})= k$. The statement of the lemma now follows from the fact that $\theta(\overline{L(G)}) = \chi (L(G)) = \chi'(G)$ is the chromatic index of $G$.
\end{proof}
\end{sloppypar}

\begin{sloppypar}
\begin{theorem}\label{thm:complements-of-line-of-triangle-free}
For every integer $k\ge 3$, {\sc Strong Clique Partition Existence} (a.k.a. {\sc Localizability}) is \NP-complete in the class of complements of line graphs
of triangle-free $k$-regular graphs.
\end{theorem}
\end{sloppypar}

\begin{proof}
The fact that the problem is in NP follows from~\Cref{cor:NP}.
Cai and Ellis showed in~\cite{journals/dam/CaiE91} that for every $k\ge 3$, it is \NP-complete to determine whether a given $k$-regular triangle-free graph is $k$-edge-colorable. Thus, if $k\ge 3$ and $G$ is a given $k$-regular triangle-free graph, \Cref{lem:complements-of-line-graphs} implies that $G$ is $k$-edge-colorable if and only if the complement of its line graph is localizable. The theorem follows.
\end{proof}

We conclude the section with some remarks related to~\Cref{thm:complements-of-line-of-triangle-free}.
The fact that testing if a given complement of a line graph is localizable is \NP-complete is in contrast with the fact that the more general property of well-coveredness can be tested in polynomial time in the class of complements of line graphs. Indeed, in order to verify whether $\overline{L(G)}$ is well-covered, it suffices to compute $G$ (using~\Cref{thm:line}), compute $\mathcal{M}_G$ (see~\Cref{def:M_G}), and test if all sets in $\mathcal{M}_G$ are of the same size.

\Cref{thm:complements-of-line-of-triangle-free} also has the following consequence.\footnote{\Cref{cor:alpha} can also be derived using a reduction from $k$-colorability of triangle-free graphs, following the approach used in the proof of~\cite[Theorem 3.3]{HMR2018}.}

\begin{sloppypar}
\begin{theorem}\label{cor:alpha}
For every integer $k\ge 3$, {\sc Strong Clique Partition Existence} (a.k.a. {\sc Localizability}) is \NP-complete in the class of graphs of independence number $k$.
\end{theorem}
\end{sloppypar}

\begin{sloppypar}
\begin{proof}
As noted in the proof of~\Cref{lem:complements-of-line-graphs},
if $G$ is a $k$-regular triangle-free graph, then $\overline{L(G)}$ has independence number $k$. Hence, \NP-hardness of the problem follows from~\Cref{thm:complements-of-line-of-triangle-free}. Membership in NP follows from the fact that for every fixed $k$, {\sc Strong Clique} is polynomial-time solvable
in the class of graphs of independence number $k$. Indeed, to test if a given clique $C$ in such a graph $G$ is strong, one only needs to enumerate all $\mathcal{O}(|V(G)|^k)$ maximal independent sets of $G$ and test if $C$ intersects each of them. Consequently, a partition of the vertex set of $G$ into $k$ strong cliques is a polynomially verifiable certificate.
\end{proof}
\end{sloppypar}

\begin{sloppypar}
The statement of~\Cref{cor:alpha} is sharp in the sense that localizable graphs of independence number at most two are recognizable in polynomial time. Indeed, using \Cref{prop:alpha-chi-bar} it is not difficult to see that a graph of independence number two is localizable if and only if its complement is a bipartite graph without isolated vertices.
Finally, we remark that the result of~\Cref{cor:alpha} contrasts with the fact that for every fixed $k$, testing if a given graph $G$ with $\alpha(G)\le k$ is well-covered can be done in polynomial time, simply by enumerating all the $\mathcal{O}(|V(G)|^k)$ independent sets and comparing any pair of maximal ones with respect to their cardinality. In particular, unless \hbox{P = NP}, one cannot obtain a simultaneous strengthening of~\Cref{thm:weakly-chordal} and~\Cref{cor:alpha} that would establish hardness of recognizing localizable graphs within the class of weakly chordal graphs of bounded independence number. Indeed, the above observation together with~\Cref{cor:perfect} implies that this problem is solvable in polynomial time.
\end{sloppypar}

\section{Graphs of small maximum degree}\label{sec:max3}

In this section, we study the six strong clique problems in graphs of small maximum degree. For the first five problems from the list, we adopt a similar approach as in~\Cref{sec:co-line}, by showing that {\sc Strong Clique Extension}, the auxiliary problem introduced in~\Cref{sec:a-problem}, is solvable in polynomial time in any class of graphs of bounded degree and, more generally, in any class of graphs of bounded clique number. Determining the complexity status of the `last' of the six strong clique problems, {\sc Strong Clique Partition Existence} (a.k.a.~{\sc Localizability}), seems to be more challenging. For this problem, we show polynomial-time solvability in the class of subcubic graphs.

\begin{proposition}\label{prop:omega-k}
For every positive integer $k$, {\sc Strong Clique Extension} is solvable in polynomial time in the class of graphs of clique number at most $k$.
\end{proposition}

\begin{proof}
Fix a positive integer $k$ and let $G$ be a given graph with clique number at most $k$. Since $k$ is a constant, the $\mathcal{O}(|V(G)|^k)$ maximal cliques of $G$ can be enumerated in polynomial time. Thus, it suffices to show that {\sc Strong Clique} can be solved in polynomial time. Indeed, if this is the case, then an efficient algorithm for testing if a given clique can be extended to a strong one can be obtained by checking if any of the maximal cliques of $G$ that contains the given clique is strong.

By~\Cref{lem:strong-cliques}, a clique $C$ in $G$ is not strong if and only if it is dominated by an independent set $I \subseteq V(G)\setminus C$. However, if this is the case, then any minimal set $I'\subseteq I$ that dominates $C$ is contained in $N_G(C)$ and contains at most $k$ vertices. Consequently, to test if $C$ is strong, it suffices to enumerate all $\mathcal{O}(|V(G)|^k)$ subsets of
$N_G(C)$ of size at most $k$ and test whether any of them is independent and dominates $C$. This completes the proof.
\end{proof}

\Cref{prop:omega-k} and \Cref{lem:reduction} imply the following.

\begin{corollary}\label{cor:omega-k}
For every positive integer $k$, each of the problems {\sc Strong Clique}, {\sc Strong Clique Existence}, {\sc Strong Clique Vertex Cover}, {\sc Strong Clique Edge Cover}, and {\sc Strong Clique Partition} is solvable in polynomial time in the class of graphs of clique number at most $k$.
\end{corollary}

Since every graph of maximum degree at most $k$ has clique number at most $k+1$, \Cref{prop:omega-k} and~\Cref{cor:omega-k} also have the following consequence.

\begin{sloppypar}
\begin{corollary}\label{cor:delta-k}
For every non-negative integer $k$, each of the problems
{\sc Strong Clique Extension}, {\sc Strong Clique}, {\sc Strong Clique Existence}, {\sc Strong Clique Vertex Cover}, {\sc Strong Clique Edge Cover}, and {\sc Strong Clique Partition} is solvable in polynomial time in the class of graphs of maximum degree at most $k$.
\end{corollary}
\end{sloppypar}

\begin{sloppypar}
We now turn to the more challenging {\sc Strong Clique Partition Existence} (a.k.a.~{\sc Localizability}) problem. Determining the complexity status of this problem in the class of $K_4$-free graphs, or more generally, in any class of graphs of bounded clique number, was stated as an open question in~\cite{HMR2018}. We are also not aware of any complexity results for this problem in classes of graphs of bounded maximum degree. This is in contrast with the fact that there is a linear-time algorithm for recognizing well-covered graphs of bounded maximum degree, due to Caro et al.~\cite{MR1640952}. A characterization of well-covered subcubic graphs was given by Ramey~\cite{MR2691098}.
\end{sloppypar}

\begin{sloppypar}
The structure of graphs of maximum degree at most two is so restricted that a polynomial-time algorithm for {\sc Strong Clique Partition Existence} (a.k.a.~{\sc Localizability}) in this graph class is straightforward. In what follows, we prove that the problem is also solvable in polynomial time in the more general class of subcubic graphs. We begin with a lemma.
\end{sloppypar}

\begin{lemma}
\label{lem:subcubic}
Let $G=(V,E)$ be a connected subcubic graph admitting a strong clique $C$ of size $3$. If there exists a strong clique $C'$ in  $G$ such that $C\neq C'$ and $C\cap C'\neq \emptyset$, then  $G$ is localizable if and only if $G\cong \overline{P_2+P_3}$.
\end{lemma}

\begin{proof}
Let $G$, $C$ and $C'$ be as in the statement.
It is straightforward to verify that $G$ is localizable if $G\cong \overline{P_2+P_3}$. Suppose that $G$ is localizable, let $\mathcal{C}$ be a partition of $V$ into strong cliques, and let $C=\{a,b,c\}$. Suppose first that $|C'|=3$. If $|C\cap C'|=1$, then the vertex in $C\cap C'$ would have degree at least 4, contrary to the assumption that $G$ is subcubic. Therefore, $|C\cap C'|=2$. We may, without loss of generality, suppose that $C\cap C'=\{a,b\}$. Let $d$ be the remaining vertex of $C'$. If there are no other vertices in $G$, then $G$ is isomorphic to the diamond graph ($K_4$ minus an edge), which is not localizable. Since $G$ is connected with $\Delta(G)=3$, it follows that one of $c$ and $d$ has a neighbor distinct from $a$ and $b$. If $c$ and $d$ are adjacent then $G\cong K_4$ contradicting the assumption that $C$ is a strong clique. Without loss of generality, we may assume that $x$ is the neighbor of $c$ different from $a$ and $b$. If $x$ and $d$ are non-adjacent, then we may extend $\{d,x\}$ to a maximal independent set $I$ in $G$ such that $I\cap C=\emptyset$, contrary to the assumption that $C$ is strong clique. We conclude that $x$ and $d$ are adjacent.  One of the cliques $C$ and $C'$ must be contained in $\mathcal{C}$, since they are the only strong cliques containing vertices $a$ and $b$. It is now clear that at least one of the cliques $\{c,x\}$ and $\{d,x\}$ must belong to $\mathcal{C}$. If $G$ has more than $5$ vertices, then since $G$ is connected with $\Delta(G)=3$, it follows that $x$ has a neighbor $y$ different from $c$ and $d$. It is clear that $y$ and $a$ are non-adjacent, since $a$ already has degree $3$. But now, any maximal independent set in $G$ containing $\{a,y\}$ is disjoint from both cliques $\{c,x\}$ and $\{d,x\}$, contrary to the assumption that one of them is in $\mathcal{C}$ and hence is strong. Thus, $G$ has exactly $5$ vertices, and it is obvious that $G\cong \overline{P_2+P_3}$.

Suppose now that $|C'|=2$.  Without loss of generality we may assume that $C\cap C'=\{a\}$ and $C'=\{a,d\}$. We claim that $d$ has degree at least 2. Suppose on contrary that $d$ has degree 1. Then the only strong clique containing $d$ is $\{a,d\}$, hence it must be contained in $\mathcal{C}$. This implies that $C\not \in \mathcal{C}$. Therefore, vertices $b$ and $c$ must be contained in some strong cliques different from $C$. By the first part of the proof, it follows that  $b$ and $c$ are contained in  strong cliques of size $2$ (otherwise we would have two strong triangles with non-empty intersection). Let $\{b,x\}$ and $\{c,y\}$ be  strong cliques containing $b$ and $c$, respectively. Since $C$ is a strong clique in $G$, it follows that $x$ and $y$ are adjacent. However, a maximal independent set containing $a$ and $x$ cannot intersect $\{c,y\}$, contradicting the assumption that $\{c,y\}$ is a strong clique. The obtained contradiction establishes our claim that $d$ has degree at least 2.

Observe that $d$ cannot be adjacent to $b$ or $c$, since $\{a,d\}$ is strong clique. Let $x$ be a neighbor of $d$ different from $a$. Since $\{a,d\}$ is a strong clique in $G$, it follows that $x$ is adjacent to both $b$ and $c$. If $d$ has third neighbor $y$, then $y$ would have to be adjacent to both $b$ and $c$ as well, contrary to the assumption that $\Delta(G)\leq 3$. Hence, $d$ has no other neighbor, and since $G$ is connected, it follows that $G$ is again isomorphic to $\overline{P_2+P_3}$.
\end{proof}

We are now ready to state the main result of this section.

\begin{theorem}
\label{thm:max3}
{\sc Strong Clique Partition Existence} (a.k.a.~{\sc Localizability})
is solvable in polynomial time in the class of subcubic graphs.
\end{theorem}

\begin{proof}
Let $G=(V,E)$ be a subcubic graph. Since a graph is localizable if and only if each of its connected components is localizable, we may assume without loss of generality that $G$ is connected and non-trivial. Since the maximum degree of $G$ is at most 3, we can find all strong cliques of $G$ in polynomial time (we need to check all subsets of size at most 3, and the subgraphs induced by their neighborhoods). If there exists a strong clique of size 4, then $G\cong K_4$, hence localizable. Otherwise, all strong cliques in $G$ are of sizes $2$ or $3$. If there exists a strong clique of size 3 with non-empty intersection with another strong clique in $G$, then by \Cref{lem:subcubic}, $G$ is localizable if and only if $G\cong \overline{P_2+P_3}$.

Hence we may assume from now on that every  strong clique of size $3$ is disjoint with every other strong clique in $G$. Therefore, if $G$ is localizable, then every strong clique of size 3 is contained in every partition of $V$ into strong cliques. We now form a new graph $G'$ formed by vertices of $G$ that do not belong to strong cliques of size $3$. Two vertices of $G'$ are adjacent if they form a strong clique of size $2$ in $G$.  If there exists a perfect matching $M$ in $G'$, then this matching, together with all strong cliques of $G$ of size $3$ is a partition of $V$ into strong cliques. Similarly, if there exists a partition of $V$ into strong cliques, then it is clear that the set of the cliques of size $2$ in this partition induces a perfect matching in $G'$. Hence $G$ is localizable if and only if $G'$ admits a perfect matching. Verifying if $G'$ has a perfect matching can be done in polynomial time (see for instance \cite{edmonds}).
\end{proof}

In the case of cubic (that is, $3$-regular) graphs, the result of \Cref{thm:max3} can be strengthened to a complete classification. This can be derived using the classification of well-covered cubic graphs due to Campbell et al.~\cite{MR1220613}. The list consists of three infinite families together with six exceptional graphs. By testing each of the graphs in the list for localizability, the classification of localizable cubic graphs given by \Cref{thm:3-regular} can be derived. For an integer $n\ge 2$, let $F_n$ denote the graph obtained as follows: take a cycle $v_1v_2\ldots v_n v_1$ of length $n$ (if $n=2$, then $v_1v_2v_1$  is a cycle of length two with two parallel edges), replace every vertex $v_i$ of the cycle by vertices $x_i,x'_i,y_i,y'_i,z_i,z'_i$ inducing the graph $F$ (see \Cref{fig:F}); replace each edge $v_iv_{i+1}$ ($i=1,\ldots,n-1$) of the cycle by an edge $x_i x'_{i+1}$, finally replace the edge $v_nv_i$ by the edge $x_nx'_1$ (see \Cref{fig:F4} for an example).

\begin{theorem}\label{thm:3-regular}
Let $G$ be a connected cubic  graph. Then,
$G$ is localizable if and only if $G$ is isomorphic to one of the graphs in the set $\{K_{3,3},K_4,\overline{C_6}\}\cup\{F_n \mid n\ge 2\}$ (see~\Cref{fig:cubic}).
\end{theorem}

\begin{figure}[!ht]
\centering
\begin{minipage}{.35\textwidth}
  \centering
  \includegraphics[width=0.6\linewidth]{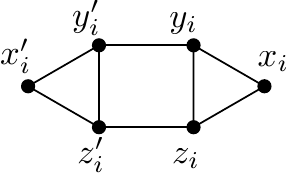}
  \captionof{figure}{The graph $F$}
  \label{fig:F}
\end{minipage}%
\hspace{1cm}
\begin{minipage}{.4\textwidth}
  \centering
  \includegraphics[width=\linewidth]{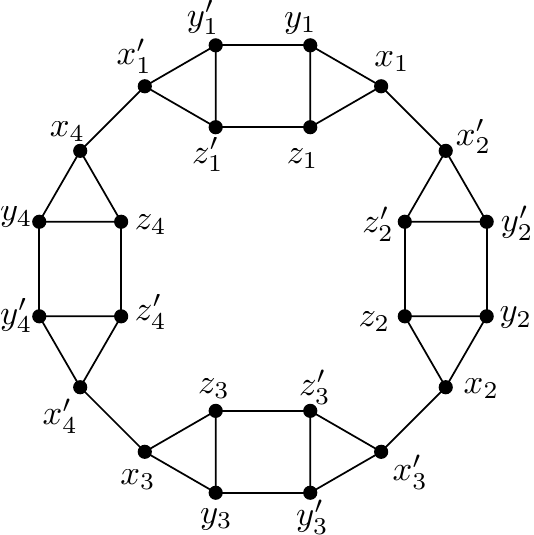}
  \captionof{figure}{The graph $F_4$}
  \label{fig:F4}
\end{minipage}
\end{figure}

\begin{figure}[!ht]
  \centering
  \includegraphics[width=\linewidth]{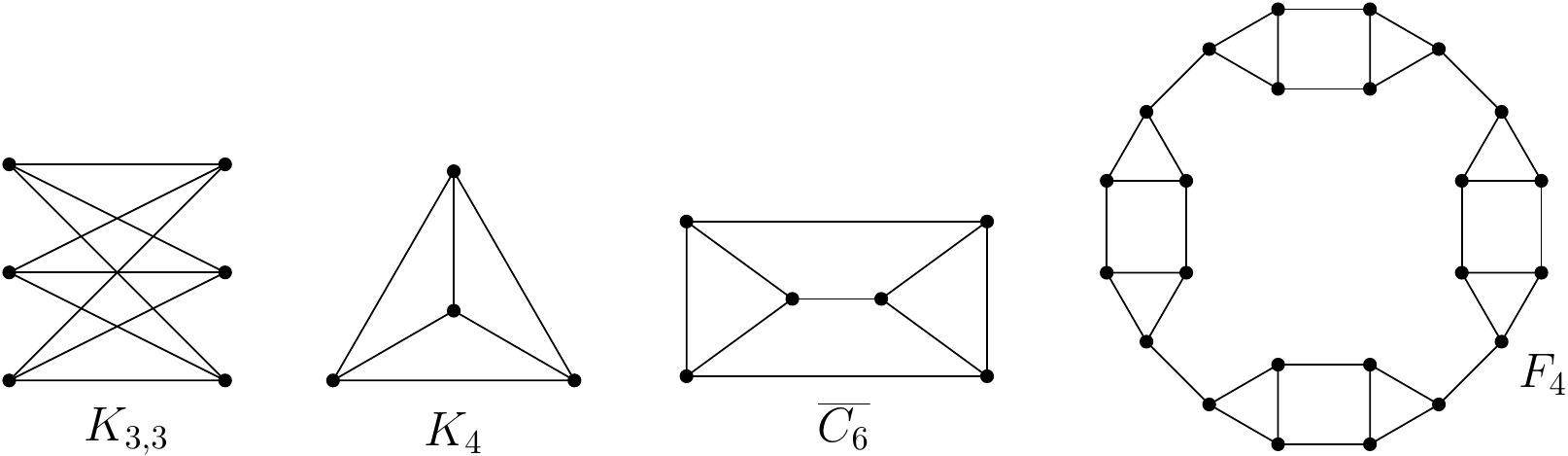}
  \captionof{figure}{Cubic localizable graphs: $K_{3,3}$, $K_4$, $\overline{C_6}$, and an infinite family $\{F_n \mid n\ge 2\}$}
  \label{fig:cubic}
\end{figure}

%See also~\cite{} for a short direct proof of this result.

We now give here a short direct proof of~\Cref{thm:3-regular} and strengthen it further by adding to the list the property that
every vertex of $G$ is contained in a strong clique. First we develop a property of strong cliques in general regular graphs.

\begin{lemma}\label{prop:trianglefree-regular}
Let $r\ge 1$ and let $G$ be a connected $r$-regular graph. Then $G$ has a strong clique of size two if and only if $G\cong K_{r,r}$.
\end{lemma}

\begin{proof}
Suppose that $G$ is a connected $r$-regular graph, and let $\{u,v\}$ be a strong clique in $G$. Since the clique $\{u,v\}$ is strong, it is a maximal clique and therefore it is not contained in any triangle. Let $A$ denote the set of neighbors of $u$ different from $v$, and let $B$ denote the set of neighbors of $v$ different from $u$. By definition, we have $A\cap B=\emptyset$. We claim that every vertex in $A$ is adjacent to every vertex in $B$. Suppose for a contradiction that some vertex $x\in A$ is not adjacent to some vertex $y\in B$. Then extending the set $\{x,y\}$ to a maximal independent set of $G$ results in a maximal independent not containing any of $u,v$, contradicting the assumption that $\{u,v\}$ is a strong clique. Therefore, every vertex in $A$ is adjacent to every vertex in $B$, as claimed.

Since $G$ is $r$-regular, $N(u) = A\cup \{v\}$, and $N(v) = B\cup \{u\}$, we have $|A|=|B|=r-1$. The connectedness of $G$ and the fact that $G$ is $r$-regular implies that $G\cong K_{r,r}$. The converse direction is immediate.
\end{proof}

\begin{theorem}\label{thm:3-regular-improved}
Let $G$ be a connected cubic  graph. Then, the following statements are equivalent:
\begin{enumerate}
  \item $G$ is localizable.
  \item Every vertex of $G$ is contained in a strong clique.
  \item $G$ is isomorphic to one of the graphs in the set $\{K_{3,3},K_4,\overline{C_6}\}\cup\{F_n \mid n\ge 2\}$ (see~\Cref{fig:cubic}).
\end{enumerate}
\end{theorem}

\begin{proof}
It immediately follows from the definition of localizable graphs that every vertex of a localizable graphs is contained in a strong clique. Therefore, statement $1$ implies statement~$2$.

It is also easy to check that each of the graphs in the set $\{K_{3,3},K_4,\overline{C_6}\}\cup\{F_n \mid n\ge 2\}$ is localizable. Indeed, each edge of $K_{3,3}$ is a strong clique; in $K_4$ the whole vertex set is a strong clique; in each of the remaining graphs each vertex is contained in a unique triangle and all these triangles are strong. Therefore, statement $3$ implies statement $1$.

Finally, we show that statement $2$ implies statement $3$. Since $G$ is cubic, it cannot have a clique of size $5$ or a strong clique of size $1$. If there exists a strong clique of size $2$ in $G$, then by~\Cref{prop:trianglefree-regular}, $G\cong K_{3,3}$. If there exists a strong clique of size $4$ in $G$, then $G\cong K_4$ since $G$ is cubic.

Suppose now that every vertex of $G$ belongs to a strong clique of size $3$. Let $C_1=\{x_1,y_1,z_1\}$ be a strong clique in $G$. All three vertices of $C_1$ cannot have a common neighbor, otherwise $C_1$ would not be a strong clique. Suppose that two vertices of $C_1$, say $x_1$ and $y_1$, have a common neighbor, say $w$, outside $C_1$. Let $w'\ne z_1$ be the remaining neighbor of $w$. Since the only clique of size $3$ that contains $w$ is $\{x_1,y_1,w\}$, it follows that $\{x_1,y_1,w\}$ is a strong clique. If $w'$ is not adjacent to $z_1$, then a maximal independent set in $G$ containing $\{w',z_1\}$ would be disjoint from $\{x_1,y_1,w\}$, contradicting the fact that $\{x_1,y_1,w\}$ is a strong clique. This implies that $w'$ is adjacent to $z_1$. However, in this case, vertex $w'$ does not belong to any triangle in $G$, which contradicts the assumption that every vertex of $G$ belongs to a strong clique of size $3$. This shows that no two vertices of $C_1$ have a common neighbor outside of $C_1$.

If the neighbors of $x_1,y_1,z_1$ outside of $C_1$ form an independent set $S$, then any maximal independent set in $G$ containing $S$ would be disjoint from $C_1$, a contradiction. Suppose now that $y'_1$ and $z'_1$ are the remaining neighbors of $y_1$ and $z_1$, and that $y'_1$ is adjacent to $z'_1$.
Since every vertex belongs to a triangle, it follows that $y'_1$ and $z'_1$ have a common neighbor, say $x'_1$. If $x_1$ is adjacent to $x'_1$ then $G\cong \overline{C_6}$. So we may assume that $x_1$ and $x'_1$ are not adjacent. It follows from the above that for any strong clique $C=\{u,v,w\}$ in $G$, there exist vertices $u',v',w'$ such that $vv', ww'\in E$, $\{u',v',w'\}$ is a strong clique and $u$ is not adjacent to $u'$.

Let $x'_2$ be the remaining neighbor of $x_1$. Since $x'_2$ belongs to a strong clique, it follows that there exist vertices $y'_2$ and $z'_2$ such that $C_2=\{x'_2,y'_2,z'_2\}$ is strong clique. Note that the fact that $G$ is cubic implies that $\{y'_2,z'_2\}\cap \{x_1,y_1,z_1,x_1',y_1',z_1'\} = \emptyset$.
Let $y_2$ and $z_2$ be the remaining neighbors of $y'_2$ and $z'_2$ respectively.
Then using the same argument as above, replacing $C$ with $C_2$ one can easily see that $y_2$ and $z_2$ are adjacent, and that they have a common neighbor, say $x_2$. If $x_2$ is adjacent with $x'_1$, then $G\cong F_2$. If $x_2$ and $x'_1$ are not adjacent, then let $x'_3$ be the remaining neighbor of $x_2$ and repeat the same argument as before. Since the graph is finite and connected, it follows that for some $n$, we will have that $x_n$ is adjacent with $x_1'$, hence $G\cong F_n$.
\end{proof}

\begin{sloppypar}
We conclude the section by noting that, as observed in~\cite{Hujd}, the equivalence between statements~1 and 2 in \Cref{thm:3-regular-improved} cannot be generalized to arbitrary regular graphs. A small $8$-regular counterexample is given by the line graph of $K_6$. Furthermore, in~\cite{Hujd} non-localizable regular graphs are constructed in which \emph{every} maximal clique is strong.
\end{sloppypar}

%====================================================================================

\section{Conclusion}

\begin{sloppypar}
In this work we performed a detailed study of the complexity of six algorithmic decision problems related to strong cliques in graphs, with a focus on the classes of weakly chordal graphs, line graphs and their complements, and subcubic graphs. Our results, summarized in~\Cref{table-results}, leave open the complexity status of
{\sc Strong Clique Edge Cover} (a.k.a.~{\sc General Partitionability})
not only for general graphs (which is a known open problem, see, e.g.,~\cite{MR3474710,MR3575013,MR2080087})
but also for the class of weakly chordal graphs.
As a corollary of our matching-based study of strong cliques in the class of complements of line graphs, we showed that {\sc Strong Clique Partition Existence} (a.k.a.~{\sc Localizability}) is \NP-complete in the class of graphs of independence number $k$, for every fixed $k\ge 3$.
As a natural question for future research, it would be interesting to perform a more detailed complexity study of problems related to strong cliques in graphs of bounded independence number.
\end{sloppypar}

%====================================================================================

\subsection*{Acknowledgements}

\begin{sloppypar}
The authors are grateful to Mark Ellingham for providing them with a copy of Ramey's PhD thesis~\cite{MR2691098} and of paper~\cite{MR1220613}.
This work is supported in part by the Slovenian Research Agency (I$0$-$0035$,
research program P$1$-$0285$ and research projects N$1$-$0032$, N$1$-$0038$, N$1$-$0062$, J$1$-$6720$, J$1$-$7051$, and J$1$-$9110$).
\end{sloppypar}

%====================================================================================

%%%%%%%%%%%%%%%%%%%%%%%%%% BIBLIOGRAPHY %%%%%%%%%%%%%%%%%%%%%
%\newpage{}
%\scriptsize{}
%\bibliography{squco-bib}

\end{document}